\documentclass[notitlepage]{amsart}

\newtheorem{pro}{Proposition}[section]
\newtheorem{teo}[pro]{Theorem}
\newtheorem{defi}[pro]{Definition}
\newtheorem{lem}[pro]{Lemma}
\newtheorem{cor}[pro]{Corollary}
\newtheorem{rk}[pro]{Remark}

\newtheorem{facts}[pro]{Facts}

\newcommand{\Ext}{\mathrm{Ext}}
\newcommand{\Tor}{\mathrm{Tor}}

\newcommand{\Hom}{\mathrm{Hom}}

\newcommand{\A}{\mathcal{A}}
\newcommand{\B}{\mathcal{B}}

\newcommand{\F}{\mathcal{F}}

\newcommand{\Le}{\mathcal{L}}

\newcommand{\X}{\mathcal{X}}
\newcommand{\Y}{\mathcal{Y}}

\newcommand{\pd}{\mathrm{pd}}
\newcommand{\op}{\mathrm{op}}

\newcommand{\Gpd}{\mathrm{Gpd}}

\newcommand{\Proj}{\mathcal{P}}

\newcommand{\Inj}{\mathcal{I}}

\newcommand{\id}{\mathrm{id}}

\newcommand{\resdim}{\mathrm{resdim}}

\newcommand{\coresdim}{\mathrm{coresdim}}

\newcommand{\Modu}{\mathrm{Mod}}
\newcommand{\Ker}{\mathrm{Ker}}

\newcommand{\GP}{\mathcal{GP}}
\newcommand{\DP}{\mathcal{DP}}

\newcommand{\GF}{\mathcal{GF}}
\newcommand{\Flat}{\mathcal{F}}

\newcommand{\GI}{\mathcal{GI}}
\newcommand{\Coker}{\mathrm{CoKer}}

\newcommand{\Ch}{\mathrm{Ch}}

\usepackage{latexsym,amssymb,amscd} 
\usepackage{amsmath}
\usepackage[all]{xypic}
\usepackage{amsthm}

\newcommand{\ortogonal}{\bot}
\newcommand{\gorro}{\wedge}

\begin{document}
\title[Some remarks on Gorenstein projective precovers]{Some remarks on Gorenstein projective precovers}
%\thanks{}

\author{V\'ictor Becerril}
\address[V. Becerril]{Centro de Ciencias Matem\'aticas. Universidad Nacional Aut\'onoma de M\'exico. 
 CP58089. Morelia, Michoac\'an, M\'EXICO}
\email{victorbecerril@matmor.unam.mx}
%\date{}
\thanks{2010 {\it{Mathematics Subject Classification}}. Primary 18G10, 18G20, 18G25. Secondary 16E10.}
\thanks{Key Words: Gorenstein projective, Contravariantly finite, Hovey  triple,  Auslander class, Duality pair}
%%%%%%%%%%%%%%%%%%%%%%%%%%%%%%%%%%%%%%%%%%%%%%%%%%%%%%%%%%%%%%%%%%%%%%%%%%%%%%
\begin{abstract} 
The existence of the Gorenstein projective precovers over arbitrary rings is an open question. In this paper, we make use of three different techniques addressing intrinsic and homological properties  of several classes of relative Gorenstein projective $R$-modules, among them including the Gorenstein projectives and Ding projectives, with the purpose of giving some situations where Gorenstein projective precovers exists. Within the development of such techniques we obtaint a family of hereditary and complete cotorsion pairs and hereditary Hovey  triples that comes from relative Gorenstein projective $R$-modules.  We also study a class  of Gorenstein projective $R$-modules relative to the Auslander class $\A _{C} (R)$ of a semidualizing $(R,S)$-bimodule $_R C_S$, where we  make use of a property of ``reduction".  
\end{abstract}  
\maketitle
%\centerline{}

\section{Introduction} 
  For $R$ an associative ring  with identity, the Gorenstein projective, injective and flat $R$-modules where introduced in \cite{EnJen93}, since then the Gorenstein homological algebra has been developing intensively as a relative version of homological algebra that replaces the classical projective, injective and flat modules and resolutions with the Gorenstein versions. While in the classical homological algebra is know the existence of projective resolutions (resp. injective and flat resolutions)  for any $R$-module without restrictions over $R$,  the situation is different for the Gorenstein version.  The question: \textit{What is the most general class of rings over which all modules have Gorenstein projective (injective) resolutions?} still open. 
    %%%%%%
  %%%%%%%
  %%%%%%%
  %%%%%%%%
  %%%%%%%%%
  %%%%
 A comprehensive response has been given by S. Estrada, A. Iacob and K. Yeomans,  who have proved that the class of Gorenstein projective $R$-modules $\GP(R)$ is special precovering on $\Modu (R)$  provided that the ring  $R$ be right coherent and left $n$-perfect \cite[Theorem 2]{EA17}. More recently, a relationship has been found between the existence of Gorenstein projective precovers and finitely presented $R$-modules with the  Second Finitistic Dimension Conjeture \cite[Theorem 5]{Estrada23}. Furthermore, a closer relationship has been presented by P. Moradifar and J.  \v{S}aroch in \cite{Saroch22}  where is proved  that \textit{contravariant finiteness of the class}\footnote{In this paper we use \textit{precovering} as synonym of \textit{contravariantly finite}.} $\GP (R) ^{< \infty} _{fin}$ (finitely generated $R$-modules of finite Gorenstein projective dimension) implies validity of the Second Finitistic Dimension Conjeture over left artinian
rings. Furthermore is proved that that contravariant finiteness of the class $\GP (R) ^{< \infty} _{fin}$ implies contravariant finiteness of the class $\Proj (R) ^{<\infty} _{fin}$, over rings where $\GP(R) _{fin}$ (the class of finitely generated Gorenstein projective $R$-modules) is contravariantly finite and the converse holds for Artin algebras. Such relation is important, since was proved by Auslander and Reiten in \cite{AuRe}  that
contravariant finiteness of the class $\Proj (R) ^{<\infty} _{fin}$, referred to as the Auslander–Reiten condition, is a sufficient condition for validity of the Second Finitistic Dimension Conjeture over an Artin algebra.

   In view of the importance contravariant finiteness of the class $\GP (R)$, we study when such class is  precovering on $\Modu (R)$. Although we know by M. Cortés-Izurdiaga, J. Šaroch \cite{Cortes} that the pair $(\GP(R) , \GP (R)^{\ortogonal})$ is always an hereditary cotorsion pair, a condition to be complete is that all projective modules are $\lambda$-pure-injective for some infinite regular cardinal $\lambda$ (in particular, if R is right $\Sigma$-pure-injective). In this paper we make use of three different ways to obtain $\GP (R)$ precovering. First, we make use of the intrinsic properties of $\GP (R)$ to see when the Holm's question \cite[Remark 4.5 (4)]{Nanquin} is fulfilled. For the second, we make use of the different generalizations of the class $\GP (R)$ analyzing when these classes match. For the third,  we make use of the technique of S. Estrada and A. Iacob \cite{Estrada23}, by finding a suitable complete hereditary cotorsion pair.

\section{Preliminaries}

In what follows, we shall work with categories of modules over an associative ring $R$ with identity. By $\Modu (R)$ and $\Modu(R^{\rm op})$ we denote the categories of left and right $R$-modules. 

Projective, injective and flat $R$-modules will be important to present some definitions, remarks and examples. The classes of projective left and right $R$-modules will be denoted by $\mathcal{P}(R)$ and $\mathcal{P}(R^{\rm op})$, respectively. Similarly, we shall use the notations $\mathcal{I}(R)$, $\mathcal{I}(R^{\rm op})$, $\mathcal{F}(R)$ and $\mathcal{F}(R^{\rm op})$ for the classes of injective and flat modules in $\Modu(R)$ and $\Modu(R^{\rm op})$, respectively. 

Concerning functors defined on modules, $\Ext^i_R(-,-)$ denotes the right $i$-th derived functor of $\Hom_R(-,-)$. If $M \in \Modu(R^{\rm op})$ and $N \in \Modu(R)$, $M \otimes_R N$ denotes the tensor product of $M$ and $N$. Recall the construction of this tensor products defines a bifunctor $- \otimes_R - \colon \Modu(R^{\rm op}) \times \Modu(R) \longrightarrow \Modu(\mathbb{Z})$, where $\Modu(\mathbb{Z})$ is the category of abelian groups. 
\subsection*{Orthogonality}

Let $\mathcal{X} \subseteq \Modu(R)$, $i \geq 1$ be a positive integer and $N \in \Modu(R)$. The expression $\Ext^i_R(\mathcal{X},N) = 0$ means that $\Ext^i_R(X,N) = 0$ for every $X \in \mathcal{X}$. Moreover, $\Ext^i_R(\mathcal{X,Y}) = 0$ if $\Ext^i_R(\mathcal{X},Y) = 0$ for every $Y \in \mathcal{Y}$. The expression $\Ext^i_R(N,\mathcal{Y}) = 0$ has a similar meaning. Moreover, by $\Ext^{\geq 1}_R(M,N) = 0$ we mean that $\Ext^i_R(M,N) = 0$ for every $i \geq 1$. One also has similar meanings for $\Ext^{\geq 1}_R(\mathcal{X},N) = 0$, $\Ext^{\geq 1}_R(N,\mathcal{Y}) = 0$ and $\Ext^{\geq 1}_R(\mathcal{X,Y}) = 0$. We can also replace $\Ext$ by $\Tor$ in order to obtain the notations for $\Tor$-orthogonality. The right orthogonal complements of $\mathcal{X}$ will be denoted by
\begin{align*}
\mathcal{X}^{\perp_i}  & := \{ M \in \Modu(R) {\rm \ : \ } \Ext^i_{R}(\mathcal{X},M) = 0 \}  & & \mbox{ and } &
\mathcal{X}^{\perp}  & := \bigcap_{i \geq 1} \mathcal{X}^{\perp_i}.
\end{align*}
The left orthogonal complements, on the other hand, are defined similarly.

\subsection*{Relative homological dimensions} 

There are homological dimensions defined in terms of extension functors. Let $M \in  \Modu (R)$ and $\mathcal{X}, \mathcal{Y} \subseteq  \Modu(R)$. The \emph{injective dimensions of $M$ and $\mathcal{Y}$ relative to $\mathcal{X}$} are defined by
\begin{align*}
\id_{\mathcal{X}}(M) &  := \inf \{ m \in \mathbb{Z}_{\geq 0} \text{ : } \Ext _R ^{\geq m+1}(\mathcal{X},M) = 0  \} & & \mbox{ and } &
\id_{\mathcal{X}}(\mathcal{Y}) & := \sup \{ \id_{\mathcal{X}}(Y) \text{ : } Y \in \mathcal{Y} \}.
\end{align*}
%The \emph{injective dimensions of $M$ and $\mathcal{X}$ relative to $\mathcal{Y}$}, denoted by $\id_{\mathcal{Y}}(M)$ and $\id_{\mathcal{Y}}(\mathcal{X})$, are defined dually. 
In the case where $\mathcal{X} =  \Modu(R)$, we write 
\begin{align*}
\id_{ \Modu(R)}(M) & = \id(M) & & \text{and} & \id_{\mathsf{Mod}(R)}(\mathcal{Y}) & = \id(\mathcal{Y})
\end{align*} 
for the (absolute) injective dimensions of $M$ and $\mathcal{Y}$.  Dually we can define the relative  dimensions $\pd _{\X} (M)$,  $\pd _{\X} (\Y)$ and $\pd (M)$, $\pd (\Y)$.

By an \emph{$\mathcal{X}$-resolution of $M$} we mean an exact complex 
\[
\cdots \to X_m \to X_{m-1} \to \cdots \to X_1 \to X_0 \to M \to 0
\]
with $X_k \in \mathcal{X}$ for every $k \in \mathbb{Z}_{\geq 0}$.  If $X_k = 0$ for $k > m$, we say that the previous resolution has \emph{length} $m$. The \emph{resolution dimension relative to $\mathcal{X}$} (or the \emph{$\mathcal{X}$-resolution dimension}) of $M$ is defined as the value
\[
\resdim_{\mathcal{X}}(M) := \min \{ m \in \mathbb{Z}_{\geq 0} \ \mbox{ : } \ \text{there exists an $\mathcal{X}$-resolution of $M$ of length $m$} \}.
\]
Moreover, if $\mathcal{Y} \subseteq  \Modu (R)$ then
\[
\resdim_{\mathcal{X}}(\mathcal{Y}) := \sup \{ \resdim_{\mathcal{X}}(Y) \ \mbox{ : } \ \text{$Y \in \mathcal{Y}$} \}
\]
defines the \emph{resolution dimension of $\mathcal{Y}$ relative to $\mathcal{X}$}. The classes of objects with bounded (by some $n \geq 0$) and finite $\mathcal{X}$-resolution dimensions will be denoted by
\begin{align*}
\mathcal{X}^\wedge_n & := \{ M \in  \Modu(R) \text{ : } \resdim_{\mathcal{X}}(M) \leq n \} & \text{and} & & \mathcal{X}^\wedge & := \bigcup_{n \geq 0} \mathcal{X}^\wedge_n.
\end{align*}
Dually, we can define \emph{$\mathcal{X}$-coresolutions} and the \emph{coresolution dimension of $M$ and $\mathcal{Y}$ relative to $\mathcal{X}$} (denoted $\coresdim_{\mathcal{X}}(M)$ and $\coresdim_{\mathcal{X}}(\mathcal{Y})$). We also have the dual notations $\mathcal{X}^\vee_n$ and $\mathcal{X}^\vee$ for the classes of $R$-modules with bounded and finite $\mathcal{X}$-coresolution dimension.

\subsection*{Approximations}

Given a class $\mathcal{X}$ of left $R$-modules  and $M \in \Modu (R)$, recall that a morphism $\varphi \colon X \to M$ with $X \in \mathcal{X}$ is an \emph{$\mathcal{X}$-precover of $M$} if for every morphism $\varphi' \colon X' \to M$ with $X' \in \mathcal{X}$, there exists a morphism $h \colon X' \to X$ such that $\varphi' = \varphi \circ h$. An $\X$-precover $\varphi $ is special if $\Coker (\varphi) = 0$ and $\Ker (\varphi) \in \X^{\ortogonal _1}$.

 A class $\mathcal{X} \subseteq \Modu(R)$ is (\emph{pre})\emph{covering} if every left $R$-module has an $\mathcal{X}$-(pre)cover. Dually, one has the notions of (\emph{pre})\emph{envelopes} and (\emph{pre})\emph{enveloping} and \emph{special (pre)enveloping} classes. 

\subsection*{Cotorsion pairs}

Two classes $\mathcal{X,Y} \subseteq \Modu (R)$ of left $R$-modules for a \emph{cotorsion pair} $(\mathcal{X,Y})$ if $\mathcal{X} = {}^{\perp_1}\mathcal{Y}$ and $\mathcal{Y} = \mathcal{X}^{\perp_1}$. 

A cotorsion pair $(\mathcal{X,Y})$ is:
\begin{itemize}
\item \emph{Complete} if $\mathcal{X}$ is special precovering, that is, for every $M \in \Modu(R)$ there is a short exact sequence $0 \to Y \to X \to M \to 0$ with $X \in \mathcal{X}$ and $Y \in \mathcal{Y}$; or equivalently, if $\mathcal{Y}$ is special preenveloping. 

\item \emph{Hereditary} if $\Ext^{\geq 1}_R(\mathcal{X,Y}) = 0$; or equivalently, if $\mathcal{X}$ is resolving (meaning that $\mathcal{X}$ is closed under extensions and kernels of epimorphisms, and contains the projective left $R$-modules) or $\mathcal{Y}$ is coresolving.  
\end{itemize}
Note that if $(\mathcal{X,Y})$ is a hereditary cotorsion pair, then $\mathcal{X} = {}^{\perp}\mathcal{Y}$ and $\mathcal{Y} = \mathcal{X}^\perp$.

\subsection*{Duality pairs}

The notion of duality pair was introduced by Holm and J{\o}rgensen in \cite{HJ09}. Two classes $\mathcal{L} \subseteq \Modu(R)$ and $\mathcal{A} \subseteq \Modu (R^{\rm op})$ of left and right $R$-modules form a \emph{duality pair} $(\mathcal{L,A})$ \emph{in $\Modu (R)$} if:
\begin{enumerate}
\item $L \in \mathcal{L}$ if, and only if, $L^+ := \Hom _{\mathbb{Z}} (L, \mathbb{Q} / \mathbb{Z}) \in \mathcal{A}$. 

\item $\mathcal{A}$ is closed under direct summands and finite direct sums. 
\end{enumerate}
One has a similar notion of duality pair in the case where $\mathcal{L}$ is a class of right $R$-modules, and $\mathcal{A}$ is a class of left $R$-modules. 

A duality pair $(\mathcal{L,A})$ is called:
\begin{itemize}
\item \emph{(co)product-closed} if $\mathcal{L}$ is closed under (co)products. 

\item \emph{perfect} if it is coproduct closed, $\mathcal{L}$ is closed under extensions and contains $R$ (regarded as a left $R$-module). 

\item \emph{complete} if $(\A, \Le)$ is also a duality pair and $(\Le, \A)$ has all the properties required to be a perfect duality pair.
\end{itemize}

%%%%%%%%%%%%%%%%%%%%%%%%%%%%%%%%%%%%
%%%%%%%%%%%%%%%%%%%%%%%%%%%%%%%%%%%%
%%%%%%%%%%%%%%%%%%%%%%%%%%%%%%%%%%%%
%%%%%%%%%%%%%%%%%%%%%%%%%%%%%%%%%%%%

%%%%%%%
%%%%%%%
%%%%%%%

\section{Gorenstein projective precovers}
We recall that the class of Gorenstein projective $R$-modules $\GP (R)$ consist of cycles of exact complexes of left projective $R$-modules which remains exact after applying the functor $\Hom _R (-, P)$ for all $P \in \Proj (R)$. Also the class of Gorenstein flat $R$-modules $\GF (R) $ consist of cycles of exact complexes of left flat $R$-modules  which remains exact  after applying the funtor  $I \otimes_R -$, for all $I \in \Inj (R^{\op})$.

From  \v{S}aroch and \v{S}t'ov\'ich\v{e}k's \cite[Remark in p.24-25]{Stov} we  know now that the condition $\GP (R) \subseteq \GF (R)$  is true if and only if $\GP (R) = \Proj \GF (R) $, which, in a way, answers Holm's question  \cite[Remark 4.5 (4)]{Nanquin}. Where the class here denoted $ \Proj \GF (R)$ is the presented in  \cite[\S 4]{Stov} called the class of  \textit{projectively coresolved Gorenstein flat $R$-modules} which consist of cycles of exact complexes of left projective $R$-modules that remains exact after applying the functor $(I \otimes _R -)$ for all $I \in \Inj (R^{\op})$. 
Now,  from \cite[Theorem 4.9]{Stov} the class $\Proj \GF (R)$ is always a special precovering class. We declare this situation as follows. 

\begin{rk} \label{no-perfect}
Let $R$ be a ring such that $\GP (R) \subseteq \GF (R)$, then the class $\GP (R)$ is special precovering. 
\end{rk}

From the previous result we can see that the conditions asked in \cite[Theorem 1]{EA17}  there are more than needed. In what follows we address a variety of conditions that make    $\GP (R)$ a class special precovering in $\Modu (R)$, for do this we use other kinds of relative projective Gorenstein modules and other particular notions that have recently appeared in the literature.. 
% every ring is GF-closed\footnote{This means that $\GF (R)$ is closed by extensions for all ring $R$.}. Thus,  is enough that the ring $R$ will be right coherent and left $n$-perfect\footnote{This means that for the ring  $R$ and $n \geq 0$ the containment  $\Proj (R) \subseteq \Flat (R) ^{\gorro} _n$ is given.}  to obtain $\GP(R)$ special precovering.   In what follows we will see that such result remains true under other conditions  less  known and that the above result can be obtained for  classes of relative Gorenstein projective $R$-modules. 

% \begin{pro} \label{GF-n-perf}
 %Consider $R$ a left $n$-perfect ring.  Then, every Gorenstein flat $R$-module has a $\GP(R)$-resolution at most $n$. That is, the containment $\GF(R) \subseteq \GP(R)  ^{\gorro} _{n}$ is given. 
 %\end{pro}
 
 %\begin{proof}
 %Indeed, from \cite[Proposition 3.1 (1)]{Estrada18} the condition $\Ext _R ^{i}(G, F ) =0 $ for all $G \in \GF (R)$  and $F \in \Flat \cap \Flat ^{\ortogonal}$ is true without conditions over the ring. Thus by using \cite[Proposition 2]{EA17} we conclude that $\GF (R) \subseteq \GP (R) ^{\gorro} _n$. 
 %\end{proof}
  
  The statements (i)-(iv) in the following Proposition are basically equivalent to Holm's question. Their proof  it follows from \cite[Theorem 2.4, Corollaries 2.5, 2,6]{XinW} and  Remark \ref{no-perfect}. It may be noted that there are other conditions involving the copure dimensions defined by Enochs and Jenda \cite{EnJen93}, which we do not deal with here. 
\begin{pro} \label{OneMain}
Given a ring  $R$ the class $\GP(R)$ is special precovering in each one of the following situations:
\begin{itemize}
\item[(i)] $\Inj (R^{\op}) \subseteq \Flat (R^{\op}) ^{\gorro}$,
\item[(ii)] $\GP (R) \subseteq \Proj  \GF (R) ^{\gorro}$,
\item[(iii)] $\GP (R) \subseteq \GF (R) ^{\gorro}$,
\item[(iv)] There exist an integer $n \geq 0$ such that $\Tor ^{R} _n (I,M)=0$ for any $I \in \Inj (R^{\op})$ and $M \in \GP (R)$,
\item[(v)] If $R$ is either a right coherent ring such that $\Flat (R) \subseteq \Proj (R) ^{\gorro}$ or a ring such that $\Inj (R^{\op}) \subseteq \Flat (R) ^{\gorro}$.
\end{itemize}

\end{pro}

Note that a situation where $R$ is close to being right coherent and left $n$-perfect \footnote{This means that for the ring  $R$ and $n \geq 0$ the containment  $\Flat (R) \subseteq \Proj (R) ^{\gorro} _n$ is given.}  \cite[Theorem 2]{EA17} arises in Proposition \ref{OneMain} (v). In the following result we give a proof with another arguments where also $\GP (R)$ is special precovering.  
The purpose of giving the details is that the arguments can be adapted to other situations, where  some kind of relative Gorenstein projective $R$-modules  match with $\GP (R)$.
\begin{pro}\label{Ding-precov}
Let $R$ be a right coherent ring, then the class of Ding projective $R$-modules $\DP (R)$ is special precovering on $\Modu (R)$. Furthermore, if $\Flat (R) \subseteq \Proj (R) ^{\gorro}$ then $\GP (R)$ is also special precovering on $\Modu (R)$.  
\end{pro}

\begin{proof}
Indeed, since $R$ is right coherent, the pair $(\Flat (R), \Inj (R^{\op}))$ is a symmetric duality pair, thus from \cite[Theorem A.6.]{BGH}\footnote{Note that for D. Bravo et. al.  the notion of duality pair differs from the one presented here.} we get that $\Proj \GF (R) = \DP (R) $. Now from \cite[Theorem 4.9]{Stov} the class $\Proj \GF (R)$ is always a special precovering class, thus $\DP (R)$ is special precovering on $\Modu (R)$. Finally, if  $\Flat (R) \subseteq \Proj (R)^{\gorro}$, we get from  \cite[Proposition 6.7]{Becerril22} that $\DP (R)$ and $\GP (R)$ coincide, therefore $\GP (R)$ is special precovering. 
\end{proof}

As is known, the class of Ding projectives  appears as a generalization of  Gorenstein projectives, while the class $\Proj \GF (R)$ comes to complement $\GP (R)$ since $\Proj \GF (R)$ is a subclass of $\GP (R)$ which consists of  Gorenstein flat $R$-modules, which makes it more versatile,  for example  to show that every ring is GF-closed  \cite{Stov}. Thus, the generalizations and variants of  $\GP (R)$ and $\GF(R)$ provide new information about them \cite{Alina20}. this is partly the reason why we address generalizations of   $\GP (R)$ and $\GF (R)$  in order to provide a better understanding of them. 

%Por otro lado pensemos en la siguiente situaci\'on que proviene de la Proposition \ref{GP=GF}. 

%\begin{pro}
%Given a symmetric duality pair $(\Le , \A)$ in $\Modu (R)$. If some $\fd (\A)$ or $\id (\Le)$ is finite, then the class $\GP _{(\Proj , \Le)}$ is precovering in $\Modu (R)$. 
%\end{pro}
 
 %\begin{proof}
 %Sea $P ^{\bullet}$ un complejo exacto de proyectivos. Usaremos el Lema de Weak dimensions and Global dimensions, que dice que si $\fd (\A)$ es finita entonces todo complejo $P ^{\bullet }$ es $\A \times _R -$exacto. Respectivamente si $\id (\Le)$ es finita, entonces tal complejo es $\Hom _R (-, \Le)$-exacto. Luego usando los argumentos de \cite[Theorem A]{Gang-Li}, podemos ver que del par de cotorsion canonico $(\Proj, \Modu (R))$ es cual es cogenerado por un conjunto (en consecuencia perfecto), obtenemos que el par de cotorsion en complejos $(\Ch_0 (P), \Ch (R))$ es tambien perfecto, por lo que para todo complejo  en $\Ch (R)$ tiene una precover en $\Ch _0(\Proj)$. Considerando el complejo $\cdots \to 0 \to M \to M \to 0 \to \cdots$ concentrado en $0$ y $-1$ siguiendo los argumentos de \cite[Theorem A]{Gang-Li}, obtenemos una $\GP _{(\Proj , \Le)}$-precover para $M$. 
 %\end{proof}

\subsection{Precovers from generalized Gorenstein $R$-modules}

In what follows, we shall consider classes $\X \subseteq \Modu(R) $ and $ \Y  \subseteq \Modu (R^{\op})$. 
The following definition of Gorenstein flat $R$-modules relative to a pair  $(\X , \Y)$ comes from  \cite[Definition 2.1]{Wang19}.

\begin{defi}\label{def:relativeGF}
An $R$-module $M$ is Gorenstein $(\X , \Y)$-flat if there exists an exact and $(\mathcal{Y} \otimes _R -)$-acyclic complex $X_\bullet \in \Ch(\mathcal{X})$ such that $M \cong Z_0(X_{\bullet})$. By $(\mathcal{Y} \otimes _R -)$-acyclic we mean that $Y \otimes _R X_{\bullet}$ is an exact complex of abelian groups for every $Y \in \mathcal{Y}$. The class of Gorenstein $(\mathcal{X,Y})$-flat $R$-modules will be denoted by $\mathcal{GF}_{(\mathcal{X, Y})}$. 
\end{defi}
%The \textit{ Gorenstein $(\X , \Y)$-flat dimension of $M \in \Modu (R)$} is defined by 
%$$ \Gfd _{(\X, \Y)} (M) := \resdim _{\mathcal{GF}_{(\mathcal{X,Y})}} (M).$$ 
%And for any class $\Z \subseteq \Modu (R)$, $\Gfd _{(\X, \Y)} (\Z) : = \sup \{\Gfd _{(\X, \Y)} (Z) : Z \in \Z\}$.\\

There are several types of classes Gorenstein $(\X , \Y)$-flat  that have been presented in recent literature. The class $\Proj \GF$ presentes above have been generalized  and extensively studied by S. Estrada, A. Iacob and M. A. P\'erez in  \cite[Definition 2.6]{Estrada18}. This class is called  \textit{projectively coresolved Gorenstein $\mathcal{B}$-flat} and here denoted $\Proj  \GF _{  \mathcal{B}}$. In such paper also is studied the class $\GF _{ \mathcal{B}}$ (in the notation of Definition \ref{def:relativeGF} is $\GF_{(\Flat, \mathcal{B})}$  where the subscript $\Flat$ denotes $\Flat (R)$) where $\mathcal{B} \subseteq \Modu (R ^{\op})$ is sometimes a \textit{semi-definable class} \cite[Definition 2.7]{Estrada18}.\\

In a near environment,  is defined by J. Gillespie in \cite{Gill19} for a complete duality pair $(\Le, \A)$ the class of \textit{Gorenstein $(\Le, \A)$-flat $R$-modules}, here denoted by $\GF _{ \A}$, and is studied \cite[\S 5]{Gill19} some of their dimensions and model structures. In such paper also is defined the class of \textit{Gorenstein $(\Le, \A)$-projective $R$-modules} here denoted by $\GP _{ \Le}$ and the class of \textit{Gorenstein $(\Le, \A)$-injective $R$-modules} here denoted by $\GI _{ \A}$. We see that when $(\Le, \A)$ is a duality pair the classes $\GP _{\Le}$, $\GI _{ \A, }$, $\GF _{ \A}$ and $\Proj \GF _{ \A}$ have an interesting interaction between them (see \cite{Becerril22}, \cite{Estrada18}, \cite{Gill19}, \cite{Gill21}, \cite{Wang20}). In fact ,the classes $\GP _{ \Le}$ and $\Proj  \GF _{\A}$ agree sometimes. This is  true  for a  symmetric duality pair $(\Le, \A)$,  from \cite[Theorem A.6.]{BGH} since the class $\Le$ is closed by pure quotients \cite[Theorem 3.1]{HJ09}. This  fact had been mentioned previously  by J. Gillespie and A. Iacob in \cite[Corollary 14]{Gill21}. We declare this result as follows. 

\begin{pro} \label{GP=GF}
Given a symmetric duality pair $(\Le, \A)$  in $\Modu (R)$ the class $\Proj \GF _{ \A}$ of projectively coresolved Gorenstein $\A$-flat $R$-modules  coincide with the class $\GP _{\Le}$ of Gorenstein $(\Le, \A)$-projective $R$-modules.
\end{pro}
 
 When for a symmetric duality pair $(\Le, \A)$ the class $\A$  is semi-definable then  $\GP _{ \Le}$ is special  precovering in the whole category $\Modu (R)$, since \cite[Theorem 2.13]{Estrada18} the class $\Proj \GF _{ \A}$ is the left part of a complete and hereditary cotorsion pair. A duality pair with such property is the definable pair mentioned in \cite[Example 12]{Gill21}, denoted $(  \langle \Flat (R)\rangle, \langle \Inj (R^{\op})\rangle)$. That is, the class of Gorenstein $(\Proj (R), \langle \Flat (R)\rangle )$-projective  $R$-modules is special precovering. We state this as follows. 
 
 \begin{pro}\label{Prop3.3}
 Consider the definable duality pair $(\langle \Flat (R)\rangle, \langle \Inj (R^{\op})\rangle)$ in $\Modu (R)$, then the  class $ \GP _{ \langle \Flat (R)\rangle}$ is the left part of a complete and hereditary cotorsion pair.  
 \end{pro}
 
   For each $n \geq 2$ there is a duality pair $( \mathcal{FP}_n\mbox{-}Flat (R),\mathcal{FP}_n\mbox{-}Inj (R^{\op}))$, which consists of  the  $FP_n$-flat left R-modules and $FP_n$-injective right $R$-modules \cite[Theorems 5.5 \& 5.6]{BP}. The right part  $\mathcal{FP}_n\mbox{-}Inj (R^{\op}) $ is definable \cite[Example 2.21. (3)]{Estrada18}.  The   classes of relative Gorenstein $R$-modules that comes from these duality pairs have been studied in  \cite{Alina20} by A. Iacob, proving in a manner similar to that used in Proposition \ref{Prop3.3}  that $\GP _{\mathcal{FP}_n\mbox{-}Flat (R)}$ is the left part of a complete and hereditary cotorsion pair, and thus is a special precovering class.  From here, we declare the following result. 
   
   \begin{pro}\label{Alina}
   If for some $n\geq 2$ the containment $\mathcal{FP}_n\mbox{-}Flat (R) \subseteq \Proj (R) ^{\gorro}$ is given then the class of Gorenstein projective $R$-modules $\GP (R)$ is special precovering. 
   \end{pro}
   
   \begin{proof}
   Note that we have $\Proj (R) \subseteq \mathcal{FP}_n\mbox{-}Flat (R) \subseteq \Proj (R) ^{\gorro}$, thus from  \cite[Proposition 6.7]{Becerril22} the classes $\GP (R)$ and $\GP _{\mathcal{FP}_n\mbox{-}Flat (R)}$ are the same. Now from \cite[Theorem 3.8]{Alina20} we get that $(\GP (R), \GP (R) ^{\ortogonal})$ is a complete and hereditary cotorsion pair. 
   \end{proof}

   %%%%Poner aqui la caracterizacion de los Proj.Coresolved con su dimension débil y relacionarla con la global de GP.  y luego abordar lo de que $id (Y) $ finito implica GP _{X,Y} precovering.

Based on the above result we are interested in  the study of the conditions over a duality pair $(\Le, \A)$, that implies that  $\GP _{ \Le}$ will be special precovering.  This conditions will be different to the notion of \textit{semidefinable} for the class $\A \subseteq \Modu (R^{\op})$.  For do this, we study interactions between the classes  of relative Gorenstein flat and relative Gorenstein projective associated to a duality pair $(\Le, \A)$. We begin by recalling some facts and a definition adapted to our setting.  
 
 \begin{defi}[GP-admissible pair] \cite[Definition 3.1]{BMS}. A pair $(\X, \Y) \subseteq \Modu (R) \times \Modu (R)$ is  \textbf{GP-admissible} if satisfies the following conditions:
\begin{enumerate}
\item $\Ext^{\geq 1}_{R}(\mathcal{X,Y}) = 0$.

\item For every  $A \in \Modu (R)$ there is an epimorphism $X \to A$ with $X \in \mathcal{X}$.

\item $\mathcal{X}$ and $\mathcal{Y}$ are closed under finite coproducts.

\item $\mathcal{X}$ is closed under extensions.

\item $\mathcal{X} \cap \mathcal{Y}$ is a relative cogenerator in $\mathcal{X}$, that is, for every  $X \in \mathcal{X}$ there is an exact sequence $0\to X \to  W \to X' \to  0$ with $X' \in \mathcal{X}$ and $W \in \mathcal{X} \cap \mathcal{Y}$. 
\end{enumerate}
\end{defi}

\begin{facts} \label{Rem1}

\begin{itemize}
\item[(i)] We see from \cite[Corollary 2.20]{Estrada18} that for a class of right $R$-modules $\mathcal{B}$ such that  $\Inj (R ^{\op}) \subseteq \mathcal{B} $ which satisfies that  $\GF _{\mathcal{B}}$ is closed under extensions, then  the pair  $(\mathcal{GF}  _{ \mathcal{B}} , \mathcal{GF}  _{ \mathcal{B}} ^{\ortogonal _1})$ is a hereditary and complete cotorsion pair. 
\item[(ii)]  Note that when $(\Le, \A)$ is a perfect duality pair, then the containments  $\Proj (R) \subseteq \Le$, and $\Inj (R^{\op}) \subseteq \A$ are given. Furthermore the pair $(\Proj (R), \Le)$ is GP-admissible (see \cite[\S 3]{Becerril22}). 
\item[(iii)] For  a complete duality pair $(\Le, \A)$ we know from \cite[Corollary 43]{Gill21}  that $(\GF _{ \A} , \GF _{ \A} ^{\ortogonal _1})$  is a perfect cotorsion pair. In consequence the class $\GF _{\A}$ is closed by extensions and by $\mathrm{(i)}$ and $\mathrm{(ii)}$ also is an hereditary cotorsion pair. 
\end{itemize}
\end{facts}

We are ready to prove a useful result, which is a generalization of \cite[Main Result]{EA17}.
\begin{teo} \label{TPrincipal}
Let $(\X, \Y) \subseteq \Modu \times \Modu (R^{\op})$  be  with the following properties.
\begin{itemize}
 \item[(i)] $\Inj (R^{\op}) \subseteq \Y$.
 \item[(ii)] The pair $(\Proj (R), \X)$ is GP-admissible.
 \item[(iii)] The class $\mathcal{GF} _{ \Y}$ is closed by extensions.
 \end{itemize}
 Assume that the inclusions  $ \GP  _{ \X} \subseteq \mathcal{GF} _{ \Y} \subseteq \GP ^{\gorro} _{ \X}$ are given. Then the class $\GP  _{ \X}$ is special precovering.
\end{teo}

\begin{proof}
Take $X \in \Modu (R)$. From Facts \ref{Rem1} (i), we have that $(\GF _{ \Y} , \GF _{ \Y} ^{\ortogonal} )$ is a hereditary and complete cotorsi\'on pair. Thus, there is an exact sequence $0 \to Y \to N \to X \to 0$ with $N \in \GF _{ \Y}$ and $Y \in \GF _{ \Y} ^{\ortogonal }$. From the containment $\GP _{ \X} \subseteq \GF _{ \Y}$, we get  $\GF _{ \Y} ^{\ortogonal} \subseteq \GP _{ \X} ^{\ortogonal }$. Since $\mathcal{GF} _{ \Y} \subseteq \GP ^{\gorro} _{ \X}$ and $(\Proj (R), \X)$ is a GP-admissible pair, then from \cite[Theorem 4.1 (a)]{BMS} for $N \in \mathcal{GF} _{ \Y} $ there is an exact sequence $0 \to W \to T \to N \to 0$ with $T \in \GP _{ \X}$ and $W \in \GP _{ \X} ^{\ortogonal}$. We can construct the following p.b digram

$$\xymatrix{ 
        W \ar@{^{(}->}[d]  \ar@{=}[r]&  W \ar@{^{(}->}[d]  \\
         A \ar@{}[dr] |{\textbf{pb}} \ar@{>>}[d]  \ar@{^{(}->}[r]  & T _{} \ar@{>>}[d] \ar@{>>}[r] & X _{}   \ar@{=}[d]  \\
  Y \ar@{^{(}->}[r]  & N   \ar@{>>}[r] & X }$$

and obtain the exact sequence $ 0 \to W \to A \to Y \to 0$, with $W, Y \in  \GP _{ \X} ^{\ortogonal }$, this implies that $A \in \GP _{ \X} ^{\ortogonal}$. Therefore we obtain the exact sequence $ 0 \to A \to T \to X \to 0$, with $T \in \GP _{ \X}$ and $A \in \GP _{ \X} ^{\ortogonal}$. That is, $T \to X$ is a special $\GP _{ \X}$-precover.
\end{proof}
As application we have the following.
\begin{cor} \label{TwoMain}
Let $R$ be  a  ring such that $\Inj (R^{\op}) \subseteq \Flat (R^{\op}) ^{\gorro}$ and $\Flat (R) \subseteq \Proj (R) ^{\gorro}$. Then,   the class of Ding projective $R$-modules $\DP (R)$ is special precovering.

\end{cor}

\begin{proof}
 We will verify the conditions of Theorem \ref{TPrincipal} with the pair $(\Flat (R), \Inj (R^{\op}))$. To this end note that $(\Proj (R), \Flat (R))$ is a GP-admissible pair and $\GP _{ \Flat }$ is precisely the class of Ding-projective $R$-modules $\DP (R)$, while $\GF _{ \Inj (R^{\op})}$ is the usual class of Gorenstein flat $R$-modules $\GF (R)$. From \cite[Corollary 2.6]{XinW} we have the containment $\GP(R)  \subseteq \GF (R)$ and always is true that $\DP (R) \subseteq \GP (R)$, therefore $\DP (R) \subseteq \GF (R)$. Also, from \cite[Proposition 3.9]{Emmanouil23}, we have that $\GF (R) \subseteq \Proj  \GF (R) ^{\gorro} $ and by \cite[Corollary 14]{Gill21} the containment $ \Proj \GF (R) \subseteq \DP (R)$ is always true. All this give us $\GF (R) \subseteq \Proj  \GF (R) ^{\gorro} \subseteq \DP (R) ^{\gorro} $. 
\end{proof}

Is know from \cite[Theorem 3.8]{Ding05} and \cite[Corollaries 4.5 and 4.6]{Gill10} that when $R$ is a Ding-Chen ring the class $\DP (R)$ of Ding projective $R$-modules is a class special precovering, the result above and Proposition \ref{Ding-precov}  exhibits other conditions of  when this occurs and shows the usefulness of Theorem \ref{TPrincipal}.

Until now, from Proposition \ref{GP=GF} we know that when $(\Le, \A)$ is a symmetric duality the containment $\GP _{\Le} \subseteq \GF _{ \A}$ is given. By Facts \ref{Rem1}, if $(\Le, \A)$ is perfect, then are fulfilled the conditions $\mathrm{(i),(ii)}$ in the previous Theorem \ref{TPrincipal}. While if $(\Le, \A)$ is  complete then $\mathrm{(iii)}$ in Theorem \ref{TPrincipal} is also true. In the present generality, we declare the following result.

\begin{teo} \label{LA-perf}
Let $(\Le, \A) $ be a complete duality pair in $\Modu (R)$ and $R$ a left n-perfect ring. Then, the class $\GP  _{\Le}$ is special precovering.
\end{teo}

\begin{proof}
Since  $(\Le, \A)$ is complete, also is symmetric and perfect. Thus, the conditions in Theorem \ref{TPrincipal} are fulfilled, except for the containment  $\mathcal{GF} _{ \A} \subseteq \GP ^{\gorro} _{ \Le}$. To prove such a containment we must assume that $R$ is $n$-perfect, that is $\Flat (R)  \subseteq \Proj (R) ^{\gorro}_n$. 

Take $G \in \mathcal{GF} _{\A}$, then there is an exact complex $\mathbf{N}$ of flat $R$-modules $(\A \otimes _R -)$-acyclic with $G = Z_0 (\mathbf{N})$. Consider a partial projective resolution of $\mathbf{N}$  
$$0 \to \mathbf{C} \to  \mathbf{P}_{n-1}  \xrightarrow{ d_{n-1} }  \mathbf{P}_{n-2} \xrightarrow{ d_{n-2} } \cdots \to   \mathbf{P}_{1} \xrightarrow{ d_{1} }  \mathbf{P}_{0}  \xrightarrow{ d_{0} } \mathbf{N} \to 0, $$ 
where  $\mathbf{C}$  is not projective (but will be an exact complex). From this, for each $j$ we have the exact sequence $0 \to C_j \to P_{n-1,j} \to \cdots \to P_{0,j} \to N_j \to 0$, with every $P _{i,j} \in \Proj (R)$, and since $N_j \in \Flat (R)  \subseteq \Proj (R) ^{\gorro}_n$ it follows that $C_j$ is projective for all $j$. Also we have the exact sequence $0 \to \Ker (d_0) \to \mathbf{P}_{0}  \to \mathbf{N} \to 0$, with $\Ker (d_0)$ exact, since $\mathbf{P}_{0}$ and $\mathbf{N}$ are exact. For each $A \in \A$ by  assumption $A \otimes _R \mathbf{N}$ is acyclic, and since $\mathbf{P}_0$ is an projective complex then $A \otimes _R \mathbf{P}_0$ is acyclic, this implies that $A \otimes _R \Ker (d_0)$ is acyclic. Repeating this procedure we obtain that $\mathbf{C}$ is an exact complex and $A \otimes_R \mathbf{C}$ is acyclic for all $A \in \A$. Therefore $ \mathbf{C}$ is an exact complex of projectives $(\A \otimes_R -)$-acyclic, i.e.  $Z_j (\mathbf{C}) \in \Proj \GF_{\A}$.  Note that for each $j$ we have the exact sequence 
$$0 \to Z _j (\mathbf{C}) \to Z_j (\mathbf{P} _{n-1}) \to \cdots Z_{j} (\mathbf{P}_0) \to Z_j (\mathbf{N}) \to 0$$ with $Z_j (\mathbf{P}_i) \in \Proj (R) $ for each $j$ and  where $Z _j (\mathbf{C}) \in \Proj \GF_{\A} = \GP _{\Le}$  by Proposition \ref{GP=GF}.  Thus, for $j = 0$ we get the desired $\mathcal{GP} _{\Le}$-resolution.
\end{proof}

\begin{rk}
Note that, from the  proof of Theorem \ref{LA-perf} we  see that when $(\Le, \A)$ is a symmetric duality pair and $R$ is $n$-perfect, then $\GF _{\A} \subseteq {[ \GP _{\Le}] }^{\gorro} _n$.
\end{rk}

When  $R$ is a right coherent ring  the duality pair  $(\Flat (R), \Inj (R^{\op}))$ es is symmetric, and since $(\Flat, \Flat ^{\ortogonal})$ is always a complete cotorsion pair, it follows that $(\Flat (R), \Inj (R^{\op}))$ is a complete duality pair. Thus, in the case when $R$ is also left $n$-perfect we recover from Theorem \ref{LA-perf}  the main result in \cite{EA17}.

%%%%%%%%%%%%%%%%%%
%%%%%%%%%%%%%%%%%%%%%
%%%%%%%%%%%%%%%%%%%%%%%

Since  the condition of being special precovering  is enough to obtain a complete and hereditary cotorsion pair \cite[Proposition 23]{Gill21}, we give in the following result a summary in such terms of what we have proved so far.
\begin{pro} \label{manypairs}
Let $R$ be a ring. The following statements are true. 
\begin{itemize}
\item[(i)] If the containment $\GP (R) \subseteq \GF (R)$ is given, then $(\GP (R), \GP (R) ^{\ortogonal})$ is a complete and hereditary cotorsion pair. 
\item[(ii)] For  $R$ a right coherent ring, the pair $(\DP (R), \DP(R) ^{\ortogonal}) $ is a complete and hereditary cotorsion pair in $\Modu (R)$. Furthermore if $\Flat (R) \subseteq \Proj (R) ^{\gorro}$ then $(\GP (R), \GP (R) ^{\ortogonal})$ is also an complete and hereditary cotorsion pair. 
\item[(iii)] If for some $n\geq 2$ the containment $\mathcal{FP}_n\mbox{-}Flat (R) \subseteq \Proj (R) ^{\gorro}$ is given, then $(\GP (R) , \GP (R) ^{\ortogonal})$ is a complete and hereditary cotorsion pair.
\item[(iv)] If $R$ is  left n-perfect and  $\Inj (R^{\op}) \subseteq \Flat (R^{\op}) ^{\gorro}$, then $(\DP (R), \DP ^{\ortogonal} ) = (\GP (R) , \GP (R) ^{\ortogonal})$ is a complete and hereditary cotorsion pair.
\item[(v)] If $(\Le, \A) $ is a complete duality pair in $\Modu (R)$ and $R$ a left n-perfect, then the pair $(\GP _{\Le}, \GP _{\Le} ^{\ortogonal})$ is a complete and hereditary cotorsion pair. 
\end{itemize}
\end{pro}

\begin{proof}
(i) It follows from Remark \ref{no-perfect}, or equivalently by Proposition {One-Main}.
 
(ii) This is Proposition \ref{Ding-precov}. 

(iii) Comes from Proposition \ref{Alina}. 

(iv) This is Corollary \ref{TwoMain}, where the equality is given from \cite[Proposition 6.7]{BMS}.

(v) It follows from Theorem \ref{LA-perf}.
\end{proof}
An interesting phenomena comes from the existence of the cotorsion pair $(\GP _{\Le}, \GP _{\Le} ^{\ortogonal})$ since it is possible to get more cotorsion pairs induced by these. In fact, in each case of the  proposition above, we will obtain a family of cotorsion pairs.
 In order to obtain such cotorsion pairs we need the following lemma. 

\begin{lem} \label{Lema01}
Let $R$ be a ring and $\Le \subseteq \Modu (R)$ such that $(\Proj (R), \Le)$ is a  GP-admissible pair.  For $M \in \Modu (R)$ and $m \in \mathbb{N}$ consider the following statements 
\begin{itemize}
\item[(i)] $M \in [\GP _{\Le} ] ^{\gorro} _{m} $,
\item[(ii)] $\Ext ^1 _R ( M,E) =0$ for all $E \in [\Proj (R)^{\gorro} _m] ^{\ortogonal_1 } \cap \GP _{\Le} ^{\ortogonal _1}$.
\end{itemize}
Then  $\mathrm{(i)}\Rightarrow \mathrm{(ii)}$. If the pair  $(\GP _{\Le} , \GP _{\Le}  ^{\ortogonal _1})$ is a complete cotorsion pair, then $\mathrm{(ii)}\Rightarrow \mathrm{(i)}$, and thus both conditions are equivalent. 
\end{lem} 

\begin{proof}
We will use the following equality $\Proj (R) ^{\gorro} _m = \GP _{\Le} ^{\ortogonal _1 } \cap [ \GP_{\Le}] ^{\gorro} _m$. From \cite[Corollaries  5.2 (b) and 4.3 (c) ]{BMS} we have the containment  $\Proj (R) ^{\gorro} _m \subseteq  \GP _{\Le} ^{\ortogonal _1 } \cap[ \GP_{\Le}] ^{\gorro} _m$. Now if $M \in  \GP _{\Le} ^{\ortogonal _1 } \cap [\GP _{\Le} ] ^{\gorro} _{m}$ then from \cite[Theorem 4.1 (b)]{BMS} there is an exact sequence $\gamma : 0 \to M \to H  \to G \to 0$ with $H \in \Proj (R) ^{\gorro} _{m}$ and $G \in \GP _{\Le}$, this implies that the previous sequence  $\gamma$ splits, since $M \in \GP _{\Le} ^{\ortogonal _1}$, and so $M $ is direct summand of $H \in \Proj(R) ^{\gorro} _m$. 

$\mathrm{(i)}\Rightarrow \mathrm{(ii)}$. By induction over $m$.  If $m =0$ is clear. Now assume that $m > 0 $. From \cite[Theorem 4.1 (a)]{BMS}  there is an exact sequence $ \theta :0 \to K \to G \to M \to 0$ with $K \in \Proj (R) ^{\gorro} _{m-1} $ and $G \in \GP _{\Le}$. By definition there is an exact sequence $0 \to G \to P \to G' \to 0$ with $P \in \Proj (R)$ and $G' \in \GP _{\Le}$. We can construct the following p.o digram
$$\xymatrix{ 
         K  \ar@{=}[d] \ar@{^{(}->}[r]  & G _{} \ar@{^{(}->}[d]  \ar@{}[dr] |{\textbf{po}} \ar@{>>}[r] & M _{}   \ar@{^{(}->}[d]    \\
  K \ar@{^{(}->}[r]  & P   \ar@{>>}[r] & Q }$$
Where $Q \in \Proj (R)^{\gorro} _m$. For $\overline{E}\in [\Proj (R)^{\gorro} _m] ^{\ortogonal _1 }$, applying $\Hom _R (-,\overline{E})$ we obtain the following commutative diagram 
$$\xymatrix{ 
         \Hom _R (M,\overline{E})   \ar@{^{(}->}[r]  & \Hom _R (G,\overline{E}) _{}     \ar[r] & \Hom _R (K,\overline{E})  _{} \ar@{=}[d]      &   \\
  \Hom _R (Q,\overline{E}) \ar@{^{(}->}[r] \ar[u]  & \Hom _R (P,\overline{E})   \ar[r] \ar[u] & \Hom _R (K,\overline{E}) \ar[r] &\Ext^1  _R (Q,\overline{E}) =0, }  $$
  Thus $\Hom _R (G,\overline{E})  \to  \Hom _R (K,\overline{E}) $ is an epimorphism for all $\overline{E}\in [\Proj (R)^{\gorro} _m] ^{\ortogonal _1}$.  Now assume that $E \in [\Proj (R)^{\gorro} _m] ^{\ortogonal _1} \cap \GP _{\Le} ^{\ortogonal _1}$. From the above and $\theta$ we have the exact sequence 
  $$\Hom _R (G,E)  \twoheadrightarrow  \Hom _R (K,E) \to  \Ext ^1  _R (M,E)  \to \Ext ^1 _R (G,E) =0 $$
  Therefore, we conclude  $\Ext _R^1 (M, E) =0$ for all $E \in [\Proj (R)^{\gorro} _m] ^{\ortogonal _1} \cap \GP _{\Le} ^{\ortogonal _1}$.
  
  $\mathrm{(ii)}\Rightarrow \mathrm{(i)}$.  Let us suppose that $(\GP _{\Le}, \GP _{\Le} ^{\ortogonal _1})$ is a complete cotorsion pair and that for $M \in  \Modu (R)$ we have  $\Ext _R^1 (M, E) =0$ for all $E \in [\Proj (R)^{\gorro} _m] ^{\ortogonal _1} \cap \GP _{\Le} ^{\ortogonal _1}$.  We will use the equivalence of  \cite[Corollary 4.3 (c)]{BMS}. For $M \in  \Modu (R)$ there is an exact sequence $0 \to M \to H \to Q \to 0$, with $H \in  \GP _{\Le} ^{\ortogonal _1}$ and $Q \in \GP _{\Le}$. Applying $\Hom _R (-, E)$ we have the exact sequence 
  $$ \Ext _R^1 (Q, E)  \to \Ext _R^1 (H, E) \to  \Ext _R^1 (M, E) =0$$
  where $\Ext _R^1 (Q, E) =0$ since $Q \in \GP _{\Le}$ and $E \in \GP _{\Le} ^{\ortogonal _1}$. Therefore $\Ext _R^1 (H, E) =0$ for all $E \in [\Proj (R)^{\gorro} _m] ^{\ortogonal _1} \cap \GP _{\Le} ^{\ortogonal _1}$, we will use this fact at the end. 
  
  From \cite[Theorem 7.4.6]{EnJen00}, for $H$ there is an exact sequence $0 \to K' \to H' \to H \to 0$, with $H' \in \Proj (R) ^{\gorro}_m$ and $K' \in [\Proj (R) ^{\gorro}_m] ^{\ortogonal _1}$. Note that $H, H' \in \GP _{\Le} ^{\ortogonal_1}$ (since $\Proj (R)  \subseteq \GP _{\Le} ^{\ortogonal _1}$ implies $\Proj (R) ^{\gorro}_m  \subseteq \GP _{\Le} ^{\ortogonal _1}$ by the dual of  \cite[Lemma 2.6]{BMS}). This implies that $K' \in \GP _{\Le} ^{\ortogonal _1}$, i.e. $K' \in [\Proj (R) ^{\gorro}_m] ^{\ortogonal _1} \cap \GP _{\Le} ^{\ortogonal _1} $, so that  $\Ext ^1 _{R} (H, K') =0$. That is $0 \to K' \to H' \to H \to 0$  splits, therefore $H \in \Proj (R) ^{\gorro}_m$. Thus, the exact sequence $0 \to M \to H \to Q \to 0$ fulfils with the conditions of \cite[Corollary 4.3 (c)]{BMS}. 
\end{proof}

With the above, we are ready to get a family of cotorsion pairs. 
\begin{teo} \label{Family}
Let $R$ be a ring and $\Le \subseteq \Modu (R)$ such that $(\Proj (R), \Le)$ is a  GP-admissible pair. If   $(\GP _{\Le} , \GP _{\Le}  ^{\ortogonal _1})$ is a complete cotorsion pair, then for each $m > 0$ the pair 
\[  \displaystyle{\left( [\GP _{\Le}] ^{\gorro} _m ,\;  [\Proj (R)^{\gorro} _m] ^{\ortogonal _1} \cap \GP _{\Le} ^{\ortogonal _1}  \right) ,}
\]
is a complete and hereditary cotorsion pair.
\end{teo}

\begin{proof}
From Lemma \ref{Lema01} we know that  following;
\begin{align*}
[\GP _{\Le}] ^{\gorro} _m & = {^{\ortogonal _1} \displaystyle{\left( [\Proj (R)^{\gorro} _m] ^{\ortogonal _1} \cap \GP _{\Le} ^{\ortogonal _1} \right) }} & \text{and} & &  [\Proj (R)^{\gorro} _m] ^{\ortogonal _1} \cap \GP _{\Le} ^{\ortogonal _1} & \subseteq { (  [\GP _{\Le}] ^{\gorro} _m )^{\ortogonal _1}}.
\end{align*}
For the other hand we have  that $\Proj (R)^{\gorro} _m \cup \GP _{\Le} \subseteq [ \GP _{\Le}] ^{\gorro} _m \cup \GP _{\Le} = [ \GP _{\Le}] ^{\gorro} _m$. Then, taking orthogonals
\[ \displaystyle{ \left(  [ \GP _{\Le}] ^{\gorro} _m     \right) ^{\ortogonal _1}  \subseteq \left(   \Proj (R)^{\gorro} _m \cup \GP _{\Le}  \right) ^{\ortogonal _1 } =  (\Proj (R)^{\gorro} _m )^{\ortogonal _1} \cap \GP _{\Le}  ^{\ortogonal _1} }.
\]
This implies that $  \left(  [ \GP _{\Le}] ^{\gorro} _m     \right) ^{\ortogonal _1} =  (\Proj (R)^{\gorro} _m )^{\ortogonal _1} \cap \GP _{\Le}  ^{\ortogonal _1} $. Since the class $ [\GP _{\Le}] ^{\gorro} _m$ is closed by kernels of epimorphisms \cite[Corollary 4.10]{BMS}, we conclude that such cotorsion pair is hereditary. 
It remains to prove that such a pair is complete. We know that $(\GP _{\Le} , \GP _{\Le}  ^{\ortogonal _1})$ is complete as well as $(\Proj (R) ^{\gorro}_m , [\Proj (R) ^{\gorro} _m] ^{\ortogonal _1})$ \cite[Theorem 7.4.6]{EnJen00}, thus for $M \in \Modu (R)$ there is an exact sequence $0 \to K \to H \to M \to 0 \to 0$ with $H \in \GP _{\Le}$  and $K\in \GP _{\Le} ^{\ortogonal _1} $. And to this $K$ there is  $0 \to K \to S \to E \to 0$ with $S \in [\Proj (R) ^{\gorro} _m ] ^{\ortogonal _1}$ and $E \in \Proj (R) ^{\gorro} _m$. Consider the following p.o. diagram 

$$\xymatrix{ 
         K _{} ^{} \ar@{^{(}->}[d]   \ar@{^{(}->}[r]  \ar@{}[dr] |{\textbf{po}}  & H _{} \ar@{^{(}->}[d]   \ar@{>>}[r] & M _{}      \ar@{=}[d]   \\
  S ^{} \ar@{^{(}->}[r]  \ar@{>>}[d]  & G   \ar@{>>}[r]  \ar@{>>}[d] & M \\
    E  \ar@{=}[r] & E    &  }$$
   Where $S \in \GP _{\Le} ^{\ortogonal _1} $, since $K \in \GP _{\Le} ^{\ortogonal _1} $ and $E \in  \Proj (R) ^{\gorro} _m \subseteq \GP _{\Le} ^{\ortogonal _1}$. That is $S \in [\Proj (R) ^{\gorro} _m ] ^{\ortogonal _1} \cap \GP _{\Le} ^{\ortogonal _1} $. Now, since $H \in \GP _{\Le} \subseteq [\GP _{\Le}] ^{\gorro} _m$ and $E \in \Proj (R) ^{\gorro} _m \subseteq  [\GP _{\Le}] ^{\gorro} _m$ it follows that $G \in  [\GP _{\Le}] ^{\gorro} _m $. Thus the exact sequence that works is $0 \to S \to G \to M \to 0$.
\end{proof}
Thus,  in each of the conditions of Proposition \ref{manypairs}, we get a family of cotorsion pairs.  As an expected consequence we also obtain a Hovey triple in $\Modu (R)$.
\begin{cor} \label{CorFam}
Let $R$ be a ring and $\Le \subseteq \Modu (R)$ such that $(\Proj (R), \Le)$ is a  GP-admissible pair. If   $(\GP _{\Le} , \GP _{\Le}  ^{\ortogonal _1})$ is a complete cotorsion pair, then for each $m > 0$ there is a hereditary Hovey triple in $\Modu (R)$ given by 
\[  \displaystyle{\left( [\GP _{\Le}] ^{\gorro} _m ,\;  \GP _{\Le} ^{\ortogonal _1}, [\Proj (R)^{\gorro} _m] ^{\ortogonal _1}   \right) }
\]
Whose homotopy category is the stable category $[\GP _{\Le}] ^{\gorro} _m \cap  [\Proj (R)^{\gorro} _m] ^{\ortogonal _1} / ( [\Proj (R)^{\gorro} _m] \cap [\Proj (R)^{\gorro} _m] ^{\ortogonal _1})$. 
\end{cor}

\begin{proof}
We will apply \cite[Theorem 1.2]{Gill15}. Note that the hereditary and complete cotorsion pairs $([\GP _{\Le}] ^{\gorro} _m ,\;  [\Proj (R)^{\gorro} _m] ^{\ortogonal _1} \cap \GP _{\Le} ^{\ortogonal _1} )$ and $(\Proj (R)^{\gorro} _m, [\Proj (R)^{\gorro} _m] ^{\ortogonal _1})$ are compatible, since from the proof of Lemma \ref{Lema01} we know that $\Proj (R) ^{\gorro} _m = \GP _{\Le} ^{\ortogonal _1 } \cap [ \GP_{\Le}] ^{\gorro} _m$. While the last assertion follows from \cite[Theorem 6.21]{Stovi14}.
\end{proof}

As stated in the introduction, we are interested in the contravariant finiteness of the class  $\GP (R)^{< \infty} _{fin}$, \footnote{Note that $\GP (R) _{\infty}$ in the notation of \cite{Saroch22}  refers to $\GP (R)^{\gorro} = \GP (R) ^{< \infty}$ in this paper.} since this implies the contravariant finiteness of the class $\Proj (R) ^{< \infty} _{fin}$ (such property implies  the second finitistic dimension conjeture  over an Artin algebra) over rings where $\GP (R) _{fin}$ is contravariantly finite.   We will prove a kind of converse over $\Modu (R)$, in the generality of assuming that $\GP_{\Le}$ is special precovering. We specify these ideas in the following result. But first we need to state a few notions adapted to our setting. 

\begin{defi} \cite[Definition 4.16]{BMS}
The \textbf{finitistic $(\Proj (R), \Le)$-Gorenstein proyective dimension} of $\Modu (R)$ is defined and denoted by 
$$\mathrm{FGPD}_{(\Proj, \Le)} (R) := \sup \{\Gpd_{\Le} (M): M \in \GP _{\Le} ^{\gorro} \}.$$
where $\Gpd_{\Le} (M) := \resdim _{\GP _{\Le}} (M)$.

The \textbf{finitistic projective dimension} of $\Modu (R)$ is $\mathrm{FPD} (R) : = \sup\{\pd (M): M \in \Proj (R) ^{\gorro}\}$.
\end{defi} 
Note that the following result have a version over each case of Proposition \ref{manypairs}.

\begin{pro} \label{FPD-pairs}
Let $R$ be a ring and $\Le \subseteq \Modu (R)$ such that $(\Proj (R), \Le)$ is a  GP-admissible pair. If   $(\GP _{\Le} , \GP _{\Le}  ^{\ortogonal _1})$ is a complete cotorsion pair and $\mathrm{FPD} (R) = t < \infty$ then $\GP _{\Le} ^{< \infty}$ is the left hand side of a complete and hereditary cotorsion pair. 
\end{pro}

\begin{proof}
Indeed, since $\mathrm{FPD} (R) = t$, from \cite[Theorem 4.23 (a)]{BMS} we get that $\mathrm{FGPD}_{(\Proj, \Le)} (R) = t < \infty$. This implies from Theorem \ref{Family} that $[\GP _{\Le}] ^{\gorro} _{t} = \GP _{\Le} ^{< \infty} $ is the left hand side of a complete and hereditary cotorsion pair. 
\end{proof}
In the following, we compare the results obtained in this work with the second main result of P. Moradifar and J.  \v{S}aroch  \cite[(2.5) Theorem]{Saroch22}. 
\begin{rk}
Let us consider $\Lambda$ an Artin algebra. We know from \cite[X, Theorem 2.4 (iv)]{Bel-Rei} that   $\GP (\Lambda)$ is the left hand side of a  complete and hereditary cotorsion pair in $\Modu (\Lambda)$, so $\GP (\Lambda)$  is contravariantly finite in  $\Modu (\Lambda)$. The following statements are true. 
\begin{itemize}
\item[(i)]  For all  $n \geq 0 $, the class $\GP (\Lambda) ^{\gorro} _{n}$ es contravariantly finite (compare with \cite[(2.5) Theorem (i)]{Saroch22}).
\item[(ii)] If the class $\Proj (\Lambda) ^{< \infty} _{fin}$ of $\Lambda$-modules  finitely generated of finite projective dimensiones contravariantly finite, by Huisgen-Zimmermann and Smal\o\;\cite{Zimer-Smalo}, we get that $\mathrm{FPD} (\Lambda) < \infty$. Thus, from Proposition \ref{FPD-pairs} we conclude that  $\GP(\Lambda) ^{\gorro}$ is contravariantly finite (compare with \cite[(2.5) Theorem (ii)]{Saroch22}).
\end{itemize}
\end{rk}

%%%%%%%%%%%%%%%%%
%%%%%%%%%%%%%%%%%%%%
%%%%%%%%%%%%%%%%%%%%%%%%
%%%%%%%%%%%%%%%%%%%%%%%%%%%%

Finally, to obtain a further application of the theory developed so far, we will address a  duality pair from which we will be able to define a interesting class of Gorenstein $R$-modules on which we will apply our theory. \subsection*{Semidualizing bimodule}

\cite{Araya,HolmWhite} Consider $R$ and $S$ fixed associative rings with identities. An $(S, R)$-bimodule $_S C _R$ is called \textbf{semidualizing} if the following conditions are satisfied. 

\begin{itemize}
\item[(a1)] $_SC$ admits a degreewise finite $S$-projective resolution.
\item[(a2)] $C_R$ admits a degreewise finite $R^{\op}$-projective resolution.
\item[(b1)] The homothety map $_S S _S \to \Hom _{R^{\op}}(C,C)$ is an isomorphism.
\item[(b2)] The homothety map  $_R R _R \to \Hom_S (C, C)$ is an isomorphism. 
\item[(c1)] $\Ext _S ^{\geq 1} (C, C) =0$.
\item[(c2)] $\Ext _{R^{\op}} ^{\geq 1} (C,C) =0$.
\end{itemize}
Wakamatsu introduced in \cite{Waka} and studied the named \textbf{generalized tilting modules}, usually called \textbf{Wakamatsu tilting modules}. Note that a bimodule $_SC_R$ is semidualizing if and only if it is Wakamatsu tilting \cite[Corollary 3.2]{Waka2}. 

Associated to a semidualizing bimodule $_S C_R$ we have the Auslander and Bass classes.

(A) The \textbf{Auslander class} $\A _C (R)$ with respect to  $_S C _R$ consists of all modules $M \in \Modu (R)$ satisfying the following conditions:
\begin{itemize}
\item[(A1)] $\Tor ^{R} _{\geq 1} (C, M )= 0 $,
\item[(A2)] $\Ext _S ^{\geq 1} (C, C \otimes _R M) =0$,
\item[(A3)] The natural evaluation  $\mu _M : M \to \Hom _S (C, C\otimes _R M) $ given by $\mu _M(x) (c) = c \otimes x$ for any $x \in M$ and $c \in C$, is an isomorphism in $\Modu (R)$.
\end{itemize} 

(B) The \textbf{Bass class} $\B _C (S)$ with respect to  $_S C _R$ consists of all modules $N \in \Modu (S)$ satisfying the following conditions:
\begin{itemize}
\item[(B1)] $\Ext _S ^{\geq 1} (C, N)= 0 $,
\item[(B2)] $\Tor _{\geq 1} ^{R} (C, \Hom _S (C,N)) =0$,
\item[(B3)] The natural evaluation  $\nu _N : C \otimes _R \Hom _S (C, N) \to N$ given by $\nu _N (c \otimes f) = f(c)$ for any $c \in C$ and $f \in \Hom _S (C,N)$, is an isomorphism in $\Modu (S)$.
\end{itemize}

 The \textbf{Auslander class} $\A _C (S^{\op})$ and the \textbf{Bass class} $\B _C (R^{\op})$ with respect to  $_S C _R$  are defined similarly.

Recently been shown in \cite{Huang} that there are two duality pairs\footnote{Note that in \cite{Huang} the bimodule is $_RC_S$, and in our text is $_S C _R$ as appear in \cite{HolmWhite}}  associated to the bimodule $_S C _R$. Rewritten for the semidualizing bimodule $_SC _R$ (this is;  in \cite{Huang} we change $R$ by $S$) as follows. 

By \cite[Theorem 3.3]{Huang}  in $\Modu (S^{\op})$, the pair
\begin{center} 
$(\A_C (S^{\op}) , \B _C (S))$
\end{center}
 is a symmetric duality pair,  and by \cite[Theorem 3.3]{Huang} in $\Modu (R)$ the pair
 \begin{center} 
 $(\A_C (R) , \B _C (R^{\op}))$ 
\end{center}
is a symmetric duality pair. Furthermore, the pairs  $(\A_C (S^{\op}) , \B _C (S))$ and $(\A_C (R) , \B _C (R^{\op}))$ are perfect duality pairs, respectively by \cite[Corollaries 3.4 \& 3.6]{Huang}. From the above, we have the following result, whose proof follows from  Proposition \ref{LA-perf}. 

\begin{pro} \label{Au-Gor}
Let $_SC _R$ be a semidualizing $(S,R)$-bimodule, and consider the associated Auslander class $\A _C(R)$. If $R$ is left n-perfect then the class   $\GP _{\A _C} (R)$ of Gorenstein $\A _C(R)$-projectives is special precovering. If in addition $\A _C(R) \subseteq \Proj (R)^{\gorro}$, then the class $\GP (R)$ is also  special precovering.
\end{pro}

Finally in (ii) of the following result, we will make use of a different technique that can also be used on other versions of Gorenstein projective $R$-modules, we refer to a kind of  ``reduction", as in \cite{Estrada23}. We will see that as an advantage of this technique, we can dispense the condition of being n-perfect in the above Proposition. For do this we recall the notion of $C$-injective $R$-modules  \cite[Definition 5.1]{HolmWhite} , which is the class $\mathcal{I}_C (R)$ consisting  by $R$-modules of the form $\Hom _S (C, I)$ where $I \in \Inj (S)$.

\begin{teo} \label{Teo-AC}
Let $_SC _R$ be a semidualizing $(S,R)$-bimodule. The class $\GP _{\A _C} (R)$ is special precovering under the following situations.
\begin{itemize}
\item[(i)] When $\A  _C (R) \subseteq \GP _{\A _C} (R)^{\gorro}$. 
\item[(ii)]If $\pd (\mathcal{I}_C (R))\leq m$ and $\id (\A _C (R))  < \infty$.
\end{itemize}
\end{teo}
\begin{proof}
(a) From  \cite[Proposition 3.9]{EnJen-n-per}, we actually know that  $\GP (R) \subseteq \A _C (R)$, in consequence $\GP _{\A _C} (R) \subseteq \A _C (R)$ since  $\Proj (R) \subseteq \A _C (R) $ from \cite[Theorem 6.2]{HolmWhite}. This give us the containment  $\GP _{\A _C} (R) \subseteq \A _C (R) \subseteq \GP _{\A _C} (R)^{\gorro} $, which implies that  all $N \in \A _C (R)$ possesses a special $\GP _{\A _C} (R)$-precover, for the above mentioned, Facts \ref{Rem1} and  \cite[Theorem 4.1 (a)]{BMS}. Also was recently proved that for a general ring the pair $(\A _C (R), \A _C (R) ^{\ortogonal})$ is a perfect cotorsion pair, thus,  by \cite[Theorem 1]{Estrada23} we get that the class $\GP _{\A _C} (R)$ is  special precovering. 

(b) Now let us suppose that $\pd (\mathcal{I}_C (R))\leq m$ and $\id (\A _C (R))  < \infty$, we will prove under this conditions that $\A  _C (R) \subseteq \GP _{\A _C} (R)^{\gorro} _{m}$, and so the result will follow from (a). To do this, take $M \in \A_{C} (R)$. By \cite[Proposition 3.12]{EnJen-n-per}  there exists an exact sequence $$0 \to M \to U^{0} \to U^1 \to \cdots $$ where each $U^{i } \in \mathcal{I}_C (R)$.

 From the inequality  $\pd (\mathcal{I}_C (R))\leq m$ and \cite[Ch. XVII, \S 1 Proposition 1.3]{Cartan} we can construct the following commutative diagram
$$\xymatrix{  
 0\ar[r] & Q_{_{}} \ar@{^{(}->}[d] \ar@{^{(}->}[r] &  P^0 _{m_{}}  \ar@{^{(}->}[d] \ar[r]& P^1 _{m_{}}  \ar@{^{(}->}[d] \ar[r] & \cdots \\
0 \ar[r]& Q_{m-1} \ar[r] \ar[d] & P^0 _{m-1} \ar[d]  \ar[r]&  P^1 _{m-1} \ar[d]  \ar[r] & \cdots  \\
    &  \vdots  \ar[d] & \vdots  \ar[d] & \vdots \ar[d] & \\
 0\ar[r] & Q_0 \ar@{>>}[d] \ar[r] & P^0 _0  \ar@{>>}[d] \ar[r]& P^1 _0  \ar@{>>}[d] \ar[r] & \cdots \\
0 \ar[r]& M \ar@{^{(}->}[r]  & U^0   \ar[r] &  U^1 \ar[r]  & \cdots  }$$
such that  $P_i ^j \in \Proj (R)$ for all $i \in \{0,1, \dots , m \}$ and all $j\geq 0$. With  $Q _i := \Ker (P^0 _i \to P^1 _i)\in \Proj (R)$ for all $i \in \{0 , \dots , m-1\}$ and 
\begin{center}
$\Lambda  : 0 \to Q \to P_m ^0 \to P_m ^1 \to \cdots$\\
 $\Theta : 0 \to Q \to Q_{m-1} \to \cdots \to Q_0 \to M \to 0$
 \end{center}
  exact complexes. Note that the complex $\Lambda$ can be completed on the left with a resolution by projective $R$-modules, and since $\id (\A _C (R) ) < \infty$, by \cite[Lemma 3.6]{Becerril21} such  exact complex is $\Hom _R (- , \A _C (R))$-exact, from where $Q \in \GP _{\A _C} (R)$. But then, from the exact complex $\Theta$ we obtain that $\resdim _{\GP _{\A _C} (R)}(M) \leq m$, that is $M \in \GP _{\A _C} (R) ^{\gorro} _m$. 
\end{proof} 
The previous result allows us to obtain another family of cotorsion pairs  and  a family Hovey triples similarly to the Corollary \ref{CorFam}. By another hand, the following result address the conditions of when the class $\GP _{\A _{C}} (R)$ coincides with the classical $\GP (R)$.

\begin{cor}
Let $_SC _R$ be a semidualizing $(S,R)$-bimodule and assume that  $\A _{C}  (R) \subseteq \Proj (R) ^{\gorro}$. Then, the class $\GP (R)$ is special precovering. 
\end{cor}

\begin{proof}
Indeed, the  condition $\A _{C}  (R) \subseteq \Proj (R) ^{\gorro}$ implies that $ \A _{C}  (R) \subseteq \GP _{\A _C} (R) ^{\gorro}$, since by  \cite[Theorem 6.2]{HolmWhite} $\Proj (R) \subseteq \A _{C}  (R)$. Thus, from Theorem  \ref{Teo-AC} (a), the class $\GP _{\A _C} (R) $ is special precovering. Currently, we also have the containments $\Proj (R) \subseteq \A _{C}  (R) \subseteq \Proj (R) ^{\gorro}$, thus from \cite[Proposition 6.7]{BMS} we obtain the equality $\GP _{\A _C} (R) = \GP  (R)$. 
\end{proof}

\bigskip

\textbf{Acknowledgements} The author want to thank professor Raymundo Bautista (Centro de Ciencias Matemáticas - Universidad Nacional Autónoma de México) for several helpful discussions on the results of this article.\\

\textbf{Funding} The author was fully supported by a CONAHCyT (actually renamed SECIHTI)  posdoctoral fellowship CVU 443002, managed by the Universidad Michoacana de San Nicolas de Hidalgo at the Centro de Ciencias Matemáticas, UNAM.

\subsection*{Declarations}\; \\

\textbf{Ethical Approval}  Not applicable.\\

\textbf{Availability of data and materials} Not applicable.\\

\textbf{Competing interests} The authors declare no competing interests.
%\textbf{Acknowledgements} The author thanks to his childs without this work could be finished more quickly.

\end{document}